\documentclass[12pt]{article}

\usepackage{geometry,amsmath,amssymb,amsthm,graphicx,epsfig}
\usepackage{enumerate}

\usepackage[colorlinks=true,linkcolor=blue,citecolor=blue]{hyperref}

\theoremstyle{plain}
  \newtheorem{theorem}{Theorem}[section]

  \newtheorem{proposition}[theorem]{Proposition}
  
\theoremstyle{definition}
  \newtheorem{example}[theorem]{Example}
  \newtheorem{remark}[theorem]{Remark}

\newcommand{\qede}{\hspace*{\fill}$\Diamond$\medskip}

\newcommand{\op}[1]{\ensuremath{\operatorname{#1}}}
\renewcommand{\Re}{\op{Re}\,}
\renewcommand{\Im}{\op{Im}\,}
\newcommand{\e}{\mathrm{e}}

\newcommand{\md}{\mathrm{d}}
\newcommand{\id}{\,\mathrm{d}}

\newcommand{\pFq}[5]{\ensuremath{{}_{#1}F_{#2} \left( \genfrac{}{}{0pt}{}{#3}{#4} \bigg| {#5} \right)}}
\newcommand{\Li}{\op{Li}}
\newcommand{\LiH}[1]{\op{Li}_{#1}\left( \tfrac12 \right)}
\newcommand{\Ti}{\op{Ti}}
\newcommand{\Gl}[2]{\op{Gl}_{#1}\left( {#2} \right)}
\newcommand{\Cl}[2]{\op{Cl}_{#1}\left( {#2} \right)}
\newcommand{\mgl}[1]{\Gl{#1}{\frac{\pi}{3}}}
\newcommand{\mcl}[1]{\Cl{#1}{\frac{\pi}{3}}}

\newcommand{\Ls}[2]{\op{Ls}_{#1}\left( {#2} \right)}
\newcommand{\LsD}[3]{\op{Ls}_{#1}^{(#2)}\left( {#3} \right)}

\title{Log-sine evaluations of Mahler measures\\
{\small This paper is dedicated to the memory of Alf van der Poorten}}

\author{Jonathan M. Borwein\thanks{Centre for Computer-assisted Research Mathematics and its Applications (CARMA), School of Mathematical and Physical Sciences, University of Newcastle, Callaghan, NSW 2308,
Australia, \texttt{jonathan.borwein@newcastle.edu.au}. Supported in part by  the Australian
Research Council and the University of Newcastle.} ~and Armin Straub\thanks{Tulane University, USA. Email: \texttt{astraub@tulane.edu}}}

\date{\ttfamily \today}

\begin{document}
\maketitle

\begin{abstract}
  We provide evaluations of several recently studied higher and multiple Mahler
  measures using log-sine integrals. This is complemented with an analysis of
  generating functions and identities for log-sine integrals which allows the
  evaluations to be expressed in terms of zeta values or more general
  polylogarithmic terms. The  machinery developed is then applied to evaluation of further
  families of multiple Mahler measures.
\end{abstract}

\section{Preliminaries}

For $k$ functions (typically Laurent polynomials) in $n$ variables the
\emph{multiple Mahler measure}, introduced in \cite{klo}, is defined by
\begin{equation*}
  \mu(P_1,P_2, \ldots, P_k)
  := \int_0^1 \cdots \int_0^1 \prod_{j=1}^k \log \left| P_j\left(e^{2\pi i t_1},
    \ldots, e^{2\pi i t_n}\right)\right| \md t_1 \md t_2 \ldots \md t_n.
\end{equation*}
When $P=P_1=P_2= \cdots =P_k$ this devolves to the \emph{higher Mahler measure},
$\mu_k(P)$, as introduced and examined in \cite{klo}. When $k=1$ both reduce
to the standard (logarithmic) \emph{Mahler measure} \cite{boyd}.

For $n=1,2, \ldots$, we consider the \emph{log-sine integrals} defined by
\begin{equation}\label{eq:slx}
  \Ls{n}{\sigma} := -\int_{0}^{\sigma}\log^{n-1} \left|2\,\sin \frac \theta 2\right|\,{\md\theta}
\end{equation}
and their moments for $k\ge 0$ given by
\begin{equation}\label{eq:slxm}
  \LsD{n}{k}{\sigma}
  := -\int_{0}^{\sigma}\theta^k\,\log^{n-1-k} \left|2\,\sin \frac \theta 2\right| \,{\md\theta}.
\end{equation}
This is the notation used by Lewin \cite{lewin, lewin2}, and the integrals in
\eqref{eq:slxm} are usually referred to as \emph{generalized log-sine
integrals}. Note that in each case the modulus is not needed for $0 \le \sigma
\le 2\pi$. Various log-sine integral evaluations may be found in Lewin's book
\cite[\S7.6 \& \S7.9]{lewin2}.

We observe that $\Ls{1}{\sigma} = -\sigma$ and that $\LsD{n}{0}{\sigma} =
\Ls{n}{\sigma}$. In particular,
\begin{equation}\label{eq:defcl2}
  \Ls{2}{\sigma} = \Cl{2}{\sigma}
  := \sum_{n=1}^\infty \frac{\sin(n\sigma)}{n^2}
\end{equation}
is the \emph{Clausen function} which plays a prominent role below. Generalized
Clausen functions will be introduced in \eqref{eq:defcl}.

\begin{remark}
  We remark that it is fitting given the dedication of this article and volume
  that Alf van der Poorten wrote the foreword to Lewin's ``bible'' \cite{lewin2}.
  In fact, he enthusiastically mentions the evaluation
  \begin{equation*}
    -\LsD{4}{1}{\frac\pi3} = \frac{17}{6480} \pi^4
  \end{equation*}
  and its relation with inverse central binomial sums. This will be explained
  in Example \ref{eg:binomialsums}.  Evaluations of log-sine integrals at
  $\pi/3$ are discussed in Section \ref{sec:lspi3}.
  \qede
\end{remark}

\begin{example}[Two classical Mahler measures revisited]\label{eg:classical}
  As we will have recourse to the methods used in this example, we reevaluate
  $\mu(1+x+y)$ and $\mu(1+x+y+z)$. The starting point is \emph{Jensen's formula}:
  \begin{equation}\label{eq:jensen}
    \int_0^1 \log \left|\alpha + e^{2\pi i\,t}\right| \id t
    = \log \left(\max\{|\alpha|,1\}\right).
  \end{equation}

  To evaluate $\mu(1+x+y)$, we use \eqref{eq:jensen} to obtain
  \begin{equation}\label{eq:mu1xy}
    \mu(1+x+y) = \int_{1/6}^{5/6} \log(2\sin(\pi y)) \id y
    = \frac{1}{\pi}\Ls2{\frac{\pi}{3}}
    = \frac{1}{\pi}\mcl{2},
  \end{equation}
  which is a form of Smyth's seminal 1981 result, see \cite[Appendix 1]{boyd}.

  To evaluate $\mu(1+x+y+z)$, we follow Boyd \cite[Appendix 1]{boyd} and observe, on applying
  Jensen's formula, that for complex constants $a$ and $b$
  \begin{equation}\label{eq:mlin}
    \mu(ax+b) = \log|a|\vee\log |b|.
  \end{equation}
  Writing $w = y/z$ we have
  \begin{align}\label{eq:mu1xyz}
    \mu(1+x+y+z) &= \mu(1+x+z(1+w))= \mu(\log|1+w|\vee\log|1+x|) \nonumber\\
    &= \frac 1{\pi^2}\,\int_0^{\pi} \md \theta \int_0^{\pi} \max\left\{\log\left(
    2\sin \frac \theta 2\right),\log\left(2\sin \frac t2\right) \right\}\id
    t \nonumber\\
    &= \frac2{\pi^2}\,\int_0^{\pi} \md \theta \int_0^\theta \log\left(2 \sin
    \frac \theta2\right) \id t \nonumber\\
    &= \frac2{\pi^2}\,\int_0^{\pi} \theta \log \left(2
    \sin \frac \theta 2\right) \id\theta \nonumber\\
    &= -\frac2{\pi^2}\LsD{3}{1}{\pi}
    = \frac72 \frac{\zeta(3)}{\pi^2}.
  \end{align}
  The final result is again due originally to Smyth.
  \qede
\end{example}

In the following developments,
\begin{equation*}
  \Li_{a_1,\ldots,a_k}(z)
  := \sum_{n_1>\cdots>n_k>0} \frac{z^{n_1}}{n_1^{a_1} \cdots n_k^{a_k}}
\end{equation*}
denotes the \emph{generalized polylogarithm} as is studied in \cite{mcv} and in
\cite[Ch. 3]{bbg}.  For our purposes, the $a_1,\ldots,a_k$ will be positive
integers. For example, $\Li_{2,1}(z) =\sum_{k=1}^\infty
\frac{z^k}{k^2}\sum_{j=1}^{k-1} \frac1j$.  In particular, $\Li_k(x) :=
\sum_{n=1}^\infty \frac{x^n}{n^k}$ is the \emph{polylogarithm of order $k$} and
\begin{equation*}
 \Ti_k(x) := \sum_{n=0}^\infty \,(-1)^n\frac{ x^{2n+1}}{(2n+1)^{k}}
\end{equation*}
the related \emph{inverse tangent of order $k$}. We use the same notation for
the analytic continuations of these functions.
% Thus, $i\Ti_3(y) = \Li_3(iy)-\Li_3(-y^2)$.

Moreover, \emph{multiple zeta values} (MZVs) are denoted by
\begin{equation*}
  \zeta(a_1,\ldots,a_k) := \Li_{a_1,\ldots,a_k}(1).
\end{equation*}
Similarly, we consider the \emph{multiple Clausen functions} ($\op{Cl}$) and
\emph{multiple Glaisher functions} ($\op{Gl}$) of depth $k$ and weight
$w=a_1+\ldots+a_k$ defined as
\begin{align}
  \Cl{a_1,\ldots,a_k}{\theta} &= \left\{ \begin{array}{ll}
    \Im \Li_{a_1,\ldots,a_k}(e^{i\theta}) & \text{if $w$ even}\\
    \Re \Li_{a_1,\ldots,a_k}(e^{i\theta}) & \text{if $w$ odd}
  \end{array} \right\}, \label{eq:defcl}\\
  \Gl{a_1,\ldots,a_k}{\theta} &= \left\{ \begin{array}{ll}
    \Re \Li_{a_1,\ldots,a_k}(e^{i\theta}) & \text{if $w$ even}\\
    \Im \Li_{a_1,\ldots,a_k}(e^{i\theta}) & \text{if $w$ odd}
  \end{array} \right\}. \label{eq:defgl}
\end{align}
As illustrated in \eqref{eq:defcl2} and later in \eqref{eq:mgl41}, the Clausen
and Glaisher functions alternate between being cosine and sine series with the
parity of the dimension.
Of particular importance will be the case of $\theta = \pi/3$ which has also
been considered in \cite{mcv}. Note that \eqref{eq:defcl} agrees with the
definition of $\op{Cl}_2$ given in \eqref{eq:defcl2}.

To conclude this section we recall the following Kummer-type polylogarithm,
\cite{lewin2,mcv}, which has been exploited in \cite{bzb} among other places:
\begin{equation}\label{eq:bw}
  \lambda_n(x) := (n-2)! \sum_{k=0}^{n-2}
  \frac{(-1)^k}{k!} \Li_{n-k}(x) \log^k |x| + \frac{(-1)^n}{n} \log^n |x|,
\end{equation}
so that
\[ \lambda_1 \left(\tfrac12\right)=\log 2, \quad
  \lambda_2 \left(\tfrac12\right)=\frac12\,\zeta(2), \quad
  \lambda_3 \left(\tfrac12\right)=\frac78\,\zeta(3), \]
and $\lambda_4\left(\tfrac 12\right)$ is the first to reveal the presence of
$\LiH{n}$.

Our other notation and usage is largely consistent with that in \cite{lewin2}
and that in the newly published \cite{NIST} in which most of the requisite
material is described.  Finally, a recent elaboration of what is meant when we
speak about evaluations and ``closed forms'' is to be found in \cite{closed}.

\section[Log-sine integrals at pi and pi/3]{Log-sine integrals at $\pi$ and $\pi/3$}

The multiple Mahler measure
\begin{equation}\label{eq:defmu1xys}
  \mu_k(1+x+y_*) := \mu(1+x+y_1, 1+x+y_2, \ldots, 1+x+y_k)
\end{equation}
was studied by Sasaki \cite[(4.1)]{sasaki}. He uses Jensen's formula
\eqref{eq:jensen} to observe that
\begin{equation}\label{eq:muk}
  \mu_k(1+x+y_*)
  = \int_{1/6}^{5/6} \log^k \left |1-e^{2\pi i\,t}\right| \id t
\end{equation}
and so provides an evaluation of $\mu_2(1+x+y_*)$. On the other hand,
immediately from $\eqref{eq:muk}$ and the definition \eqref{eq:slx} of the
log-sine integrals we have:

\begin{theorem}\label{thm:muk}
  For positive integers $k$,
  \begin{equation}\label{eq:mukb}
    \mu_k(1+x+y_*) = \frac1\pi\Ls{k+1}{\frac\pi3} - \frac1\pi\Ls{k+1}{\pi}.
  \end{equation}
\end{theorem}

In Sections \ref{sec:lspi} and \ref{sec:lspi3} we will cultivate
Theorem \ref{thm:muk} by showing how to recursively evaluate the
log-sine integrals at $\pi$ and $\pi/3$ respectively. In view of
Theorem \ref{thm:muk} this then provides evaluations of all multiple
Mahler measures $\mu_k(1+x+y_*)$ as is made explicit in Section
\ref{sec:mu1xys}.

Further Mahler measure evaluations given later in this paper will further
involve the generalized log-sine integrals, defined in \eqref{eq:slxm}, at
$\pi$. These are studied in Section \ref{sec:lsxpi}.

\subsection[Log-sine integrals at pi]{Log-sine integrals at $\pi$}
\label{sec:lspi}

First, \cite[Eqn (8)]{lewin} provides
\begin{equation}\label{eq:lsatpirec}
  \Ls{n+2}{\pi} = (-1)^n n!\, \left( \pi\, \alpha(n+1)
  + \sum _{k=1}^{n-2} \frac{(-1)^k}{(k+1)!} \,\alpha  (n-k) \Ls{k+2}{\pi}  \right),
\end{equation}
where $\alpha(m) = (1-2^{1-m}) \zeta(m)$.
Note that $\alpha(1)=0$ while for $m\ge2$
\begin{equation*}
  \alpha(m) = -\Li_m(-1)
  = \sum _{k=1}^{\infty }{\frac { \left( -1 \right) ^{k+1}}{{k}^{m}}}.
\end{equation*}
This is a consequence of the exponential generating function \cite[Eqn.
(7.109)]{lewin2} for the requisite log-sine integrals:
\begin{equation}\label{eq:lsatpiogf}
  -\sum _{m=0}^{\infty }\Ls{m+1}{\pi} \frac {x^m}{m!}
  = \pi \frac{\Gamma\left(1+x\right)}{\Gamma^2\left(1+\frac x2\right)}
  = \pi \binom{x}{x/2}.
\end{equation}
This will be revisited and explained in Section \ref{sec:mu1x}.

\begin{example}[Values of ${\rm Ls}_{n}(\pi)$]\label{ex:pi}
  We have $\Ls{2}{\pi}=0$ and
  \begin{align*}
    -\Ls{3}{\pi} &= {\frac {1}{12}}\,\pi^3, \\
    \Ls{4}{\pi} &= \frac32\pi\,\zeta(3), \\
    -\Ls{5}{\pi} &= \frac{19}{240}\,\pi^5, \\
    \Ls{6}{\pi} &= {\frac{45}{2}}\,\pi \,\zeta(5)+\frac54\,\pi^3\zeta(3), \\
    -\Ls{7}{\pi} &= {\frac {275}{1344}}\,\pi^7+{\frac{45}{2}}\,\pi \, \zeta^2(3), \\
    \Ls{8}{\pi} &= {\frac {2835}{4}}\,\pi \,\zeta(7)+ {\frac{315}{8}}\,\pi^3\zeta(5)
      +{\frac{133}{32}}\,\pi^5\zeta(3),
  \end{align*}
  and so forth. The fact that each integral is a multi-variable rational
  polynomial in $\pi$ and zeta values follows directly from the recursion
  \eqref{eq:lsatpirec}. Alternatively, these values may be conveniently obtained
  from \eqref{eq:lsatpiogf} by a computer algebra system as the following snippet
  of \emph{Maple} code demonstrates:\\
  \texttt{for k to 7 do simplify(subs(x=0,diff(Pi*binomial(x,x/2),x\$k))) od}
  \qede
\end{example}

\subsection[Log-sine integrals at pi/3]{Log-sine integrals at $\pi/3$}
\label{sec:lspi3}

In this section, we turn to the log-sine integrals integrals at
$\pi/3$. It is shown in \cite{logsin-evaluations} that the log-sine
integrals $\LsD{n}{k}{\tau}$ can be evaluated in terms of zeta
values with the addition of multiple Clausen and Glaisher functions
at $\tau$.  The gist of the technique originates with Fuchs
(\cite{fuchs}, \cite[\S7.10]{lewin2}). In the case $\tau=\pi/3$ the
resulting evaluations usually allow considerable reductions. This is
because the basic sixth root of unity $\omega=\e^{i\pi/3}$ satisfies
$\overline{\omega}=\omega^2$.  As a consequence, the log-sine
integrals $\LsD{n}{k}{\pi/3}$ are more tractable than those at other
values; which fact we illustrate next.

\begin{example}[Reducibility]\label{eg:lstau}
  Proceeding as in \cite{logsin-evaluations}, in addition to
  $\LsD{n}{n-1}{\tau}=-\tau^n/n$ and $\Ls{2}{\tau}=\Cl{2}{\tau}$, we have
  \begin{align*}
    -\Ls{3}{\tau}
    &= 2\Gl{2,1}{\tau} + \frac{1}{12} \tau(3\pi^2-3\pi\tau+\tau^2)\\
    \LsD{3}{1}{\tau}
    &= \Cl{3}{\tau} + \tau\Cl{2}{\tau} - \zeta(3),
  \end{align*}
  as well as
  \begin{align*}
    -\Ls{4}{\tau}
    &= -6\Cl{2,1,1}{\tau} + \frac32\Cl{4}{\tau} + \frac32(\pi-\tau)\Cl{3}{\tau}
    - \frac34(\pi-\tau)^2\Cl2{\tau} - \frac32\pi\zeta(3),\\
    \LsD{4}{1}{\tau}
    &= \frac{1}{180}\pi^4 - \frac{1}{16}\tau^4 + \frac16\pi\tau^3
    - \frac18\pi^2\tau^2 - 2\Gl{3,1}{\tau} - 2\tau\Gl{2,1}{\tau}, \\
    \LsD{4}{2}{\tau}
    &= -2\Cl{4}{\tau} + 2\tau\Cl{3}{\tau} + \tau^2\Cl{2}{\tau}.
  \end{align*}

  In the case $\tau=\pi/3$ these evaluations can be further reduced as will be
  shown in Example \ref{ex:lspi3}.
  On the other hand, it appears that, for instance, $\Gl{2,1}{\tau}$ is not
  reducible even for the special values $\tau=\pi/2$ or $\tau=2\pi/3$. Here,
  reducible means expressible in terms of multi zeta values and Glaisher (resp. Clausen)
  functions of the same argument and lower weight.  Yet, $\Gl{2,1}{2\pi/3}$ is
  reducible to one-dimensional polylogarithmic terms at different arguments as
  will be shown in \eqref{eq:ls3at2pi3}.

  More generally, in \cite{logsin2} explicit reductions for all weight-four-or-less
  polylogarithms are given.
  \qede
\end{example}

\begin{example}[Values of $\Ls{n}{\pi/3}$]\label{ex:lspi3}
  The following evaluations may be obtained with the help of the
  implementation\footnote{available for download from
  \url{http://arminstraub.com/pub/log-sine-integrals}} accompanying
  \cite{logsin-evaluations}:
  \begin{align*}
    \Ls{2}{\frac\pi3} &= \mcl{2}, \\
    -\Ls{3}{\frac\pi3} &= \frac{7}{108}\, \pi^3, \\
    \Ls{4}{\frac\pi3} &= \frac12\pi\,\zeta(3)+\frac 92\,\mcl{4}, \\
    -\Ls{5}{\frac\pi3} &= \frac{1543}{19440}\pi^5 - 6\mgl{4,1}, \\
    \Ls{6}{\frac\pi3} &= \frac{15}2\pi \,\zeta(5) + \frac{35}{36}\,\pi^3
      \zeta(3) + \frac{135}{2}\,\mcl{6}, \\
    -\Ls{7}{\frac\pi3} &= \frac{74369}{326592}\pi^7 + \frac{15}{2}\,\pi
      \zeta(3)^2 - 135\,\mgl{6,1}, \\
    \Ls{8}{\frac\pi3} &= \frac{13181}{2592}\pi^5\zeta(3) + \frac{1225}{24}\pi^3\zeta(5)
      + \frac{319445}{864}\pi\zeta(7) \\ &+ \frac{35}{2}\pi^2\mcl{6}
      + \frac{945}{4}\mcl{8} + 315\mcl{6,1,1},
  \end{align*}
  and so forth, where we note that each integral is a multivariable rational
  polynomial in $\pi$ as well as $\op{Cl}$, $\op{Gl}$, and zeta values.

  The first presumed-irreducible value that occurs is
  \begin{align}\label{eq:mgl41}
    \mgl{4,1} &= \sum_{n=1}^\infty \frac{\sum_{k=1}^{n-1}\frac{1}{k}}{n^4} \,
      \sin\left( \frac{n \pi} 3\right) \nonumber\\
    &= \frac{3341}{1632960} \pi^5 - \frac{1}{\pi}\zeta(3)^2
    - \frac{3}{4\pi} \sum_{n=1}^\infty \frac{1}{\binom{2n}{n} n^6}.
  \end{align}
  The final evaluation is described in \cite{mcv}.  Extensive computation
  suggests it is not reducible in the sense of Example \ref{eg:lstau}.  Indeed,
  conjectures are made in \cite[\S5]{mcv} for the number of irreducible
  Clausen and Glaisher values at each depth.
  \qede
\end{example}

\begin{example}[Central binomial sums]\label{eg:binomialsums}
  As suggested by \eqref{eq:mgl41}, the log-sine integral $\LsD{n}{1}{\pi/3}$
  has an appealing evaluation in terms of the central binomial sum
  \begin{equation*}
   \mathcal{S}_\pm(n) := \sum_{k=1}^\infty \frac{\left(\pm 1\right)^{k+1}}{\binom{2k}{k} k^n}
  \end{equation*}
 which is given by
  \begin{equation}\label{eq:lsbin}
    -\LsD{n+2}{1}{\frac\pi3} =  n!\left(-\frac{1}2\right)^{n}\, \mathcal{S}_+(n+2).
  \end{equation}
  This is proven in \cite[Lemma 1]{mcv},   in connection with a study of
  Ap\'ery-like sums --- of which the value  $\frac 52 \mathcal{S}_{-}(3)
  =\zeta(3)$ plays a role in Ap\'ery's proof of the later's irrationality.
   The story  of Ap\'ery's proof
  is charmingly described in Alf van der Poorten's most cited paper \cite{alf}.

  Comtet's evaluation $\mathcal{S}_{+}(4) =\frac{17}{36}\,\zeta(4)$ thus also
  evaluates  $\LsD{4}{1}{\frac\pi3}=-\frac{17\pi^4}{6480}$, while the classical
  arcsin series gives  $\LsD{2}{1}{\frac\pi3}=-\frac{\pi^2}{18}$.
  We recall from \cite{mcv} that, for instance,
  \begin{equation*}
    \mathcal{S}_+(8) = \frac{3462601}{2204496000}\pi^8 + \frac19\pi^2\zeta(3)^2
      - \frac{38}{3}\zeta(3)\zeta(5) - \frac{14}{15}\zeta(5,3) - 4\pi\mgl{6,1}.
  \end{equation*}
  Thus, apart from MZVs, $\mathcal{S}_+(8)$ involves the same Clausen value
  $\mgl{6,1}$ as appears  in $\Ls{7}{\frac\pi3}$ (and hence $\mu_6(1+x+y_*)$).
  In other words, $\mu_6(1+x+y_*)$ can be written  entirely in terms of MZVs
  and  $\mathcal{S}_+(8)$.  This is true for the other cases in Example
  \ref{ex:mu} as well: $\mu_k(1+x+y_*)$ can be written in terms of MZVs as well as
  $\mathcal{S}_+(k+2)$ for $k\le6$.
  \qede
\end{example}

\subsection[Generalized log-sine integrals at pi]{Generalized log-sine integrals at $\pi$}
\label{sec:lsxpi}

Following \cite{logsin-evaluations}, we demonstrate how the generalized
log-sine integrals $\LsD{n}{k}{\pi}$ may be extracted from a generating
function given in Theorem \ref{thm:gfb}.  As Lewin \cite[\S7.9]{lewin2}
sketches, at least for small values of $n$ and $k$, these log-sine integrals at
$\pi$ have closed forms involving zeta values and Kummer-type constants such as
$\Li_4(1/2)$. This will be made more precise in Remark \ref{rk:lsxpiLi}. We
start with the generating function identity
\begin{align}\label{ex:lsmom2}
 -\sum_{n,k \ge 0} \LsD{n+k+1}{k}{\pi} & \frac{\lambda^n}{n!}\frac{(i\mu)^k}{k!}
 = \int_0^\pi \left( 2\sin\frac\theta2 \right)^\lambda \e^{i\mu\theta} \id\theta \nonumber\\
 % &= i\e^{i\pi \frac\lambda2}\,B\left(1+\lambda,\mu-\frac\lambda2\right)- i{\e^{i\pi \mu}}
 % \int_{0}^{\pi/2}\!\frac{\tan^{2\mu} \frac\theta2}{\left(\frac 12\sin\theta\right)^{1+\lambda}} \id\theta.
 &= i\e^{i\pi \frac\lambda2}\,B\left(\mu-\frac\lambda2,1+\lambda\right)- i{\e^{i\pi \mu}}
 B_{1/2}\left(\mu-\frac\lambda2,-\mu-\frac\lambda2\right)
\end{align}
given in \cite{lewin2}. Here $B_x$ is the \emph{incomplete Beta} function.
With care --- because of the singularities at zero --- \eqref{ex:lsmom2} can be
differentiated as needed as suggested by Lewin.

Using the identities, valid for $a,b>0$ and $0<x<1$,
\begin{align*}
  B_x(a,b)
  = \frac {x^a (1-x)^{b-1}}{a}\pFq21{1-b,1}{a+1}{\frac x{x-1}}
  =\frac {x^a (1-x)^b}{a}\pFq21{a+b,1}{a+1}{x},
  % = \frac {x^a}{a}\pFq21{a,1-b}{a+1}{x}
\end{align*}
found for instance in \cite[\S8.17(ii)]{NIST},
the generating function \eqref{ex:lsmom2} can be rewritten as
\begin{align*}
  -\sum_{n,k \ge 0}\LsD{n+k+1}{k}{\pi}\frac{\lambda^n}{n!}\frac{(i\mu)^k}{k!}
    &= i \e^{i\pi\frac\lambda2} \left( B_1\left(\mu-\frac\lambda2,1+\lambda\right)
    - B_{-1}\left(\mu-\frac\lambda2,1+\lambda\right) \right).
\end{align*}
Upon expanding the right-hand side this establishes the following
computationally more accessible form given in \cite{logsin-evaluations}:

\begin{theorem}[Generating function for $\LsD{n+k+1}{k}{\pi}$]\label{thm:gfb}
  For $2|\mu| <\lambda <1$ we have
  \begin{align}\label{eq:gfsum}
    -\sum_{n,k \ge 0}\LsD{n+k+1}{k}{\pi}\frac{\lambda^n}{n!}\frac{(i\mu)^k}{k!}
        &= i \sum_{n\ge0} (-1)^n\,\binom{\lambda}{n}
    \frac{\e^{i\pi\frac\lambda2} -(-1)^n \e^{i\pi\mu}}{\mu-\frac\lambda2+n}.
  \end{align}
\end{theorem}

The log-sine integrals $\LsD{n}{k}{\pi}$ can be quite comfortably extracted
from \eqref{eq:gfsum} by appropriately differentiating its right-hand side. For
that purpose it is very helpful to observe that
\begin{equation}\label{eq:diffbinom}
  \frac{(-1)^\alpha}{\alpha!} \left( \frac{\md}{\md\lambda}\right)^{\alpha} \binom{\lambda}{n} \bigg|_{\lambda=0}
  = \frac{(-1)^n}{n} \,\sum_{n>i_1>i_2>\ldots>i_{\alpha-1}}
  \frac{1}{i_1 i_2 \cdots i_{\alpha-1}}.
\end{equation}
Fuller theoretical and computational details are given in
\cite{logsin-evaluations}.

The general process is now exemplified for the cases $\LsD{4}{2}{\pi}$ and
$\LsD{5}{1}{\pi}$.

\begin{example}[$\LsD{4}{k}{\pi}$ and $\LsD{5}{k}{\pi}$]\label{eg:lslow}
  In order to find $\LsD{4}{2}{\pi}$ we differentiate \eqref{eq:gfsum} once
  with respect to $\lambda$ and twice with respect to $\mu$. To further
  simplify computation, we take advantage of the fact that the result will be
  real which allows us to neglect imaginary parts:
  \begin{align}\label{eq:ls42}
    -\LsD{4}{2}{\pi} &= \frac{\md^2}{\md\mu^2} \frac{\md}{\md\lambda} i \sum_{n\ge0} \binom{\lambda}{n}
    \frac{(-1)^n \e^{i\pi\frac\lambda2} - \e^{i\pi\mu}}{\mu-\frac\lambda2+n}\bigg|_{\lambda=\mu=0} \nonumber\\
    &= 2\pi \sum_{n\ge1} \frac{(-1)^{n+1}}{n^3}
    = \frac32\,\pi\zeta(3).
  \end{align}
  In the second step we were able to drop the term corresponding to $n=0$ because its
  contribution $-i\pi^4/24$ is purely imaginary.

  Similarly, writing $H_{n-1}^{(1,1)} = \sum_{n>n_1>n_2} \frac{1}{n_1 n_2}$, we
  obtain $\LsD{5}{1}{\pi}$ as
  \begin{align}\label{eq:ls51}
     -\LsD{5}{1}{\pi}
     &= \frac34 \sum_{n\ge1} \frac{6(1-(-1)^n)}{n^5} - \frac{\pi^2}{n^3} + \frac{8(1-(-1)^n)}{n^4}
     \left( n H_{n-1}^{(1,1)} - H_{n-1} \right) \nonumber\\
     &= \frac92 \left(\zeta(5) - \Li_{5}(-1)\right) - \frac34 \pi^2 \zeta(3) \nonumber\\
     &+ 6\left( \Li_{3,1,1}(1) - \Li_{3,1,1}(-1) - \Li_{4,1}(1) + \Li_{4,1}(-1) \right) \nonumber\\
     &= 2\,\lambda_5\left(\tfrac12\right) - \frac34\pi^2\zeta(3) - \frac{93}{32}\zeta(5).
  \end{align}
  Here $\lambda_5$ is as defined in \eqref{eq:bw}.  Further such evaluations
  include
  \begin{align}
    -\LsD{4}{1}{\pi} &= 2\lambda_4\left(\tfrac12\right) - \frac{19}{8}\zeta(4), \label{eq:ls41}\\
    -\LsD{5}{2}{\pi}
    % &= 8\pi\,\LiH{4}  + \frac{\pi}{3}\, \log^4 2 +7 \pi\zeta(3)\log 2
    % - \frac{7}{120}\pi^5-\frac{\pi^3}3 \, \log ^{2}2 \nonumber\\
    &= 4\pi\,\lambda_4\left(\tfrac12\right)-\frac{3}{40} \,\pi^5, \label{eq:ls52}\\
    -\LsD{5}{3}{\pi} &= \frac94\pi^2\zeta(3) - \frac{93}{8}\zeta(5).
  \end{align}
  $\LsD{5}{2}{\pi}$ has also been evaluated in \cite[Eqn. (7.145)]{lewin2} but
  the exact formula was not given correctly.
  \qede
\end{example}

\begin{remark}\label{rk:lsxpiLi}
  From the form of \eqref{eq:gfsum} and \eqref{eq:diffbinom} we can see
  that the log-sine integrals $\LsD{n}{k}{\pi}$ can be expressed in terms of
  $\pi$ and the polylogarithms $\Li_{n,\{1\}_m}(\pm 1)$.  Further, the duality
  results in \cite[\S 6.3, and Example 2.4]{b3l} show that the terms
  $\Li_{n,\{1\}_m}(-1)$ will produce explicit multi-polylogarithmic
  values at $1/2$.
  \qede
\end{remark}

The next example illustrates the rapidly growing complexity of these integrals,
especially when compared to the evaluations given in Example \ref{eg:lslow}.

\begin{example}[$\LsD{6}{k}{\pi}$ and $\LsD{7}{3}{\pi}$]\label{ex:ls73}
  Proceeding as in Example \ref{eg:lslow} and writing
  \[ \Li^\pm_{a_1,\ldots,a_n} = \Li_{a_1,\ldots,a_n}(1) - \Li_{a_1,\ldots,a_n}(-1) \]
  we find
  \begin{align}
    -\LsD{6}{1}{\pi}
    & = -24\Li^\pm_{3,1,1,1} + 24\Li^\pm_{4,1,1} - 18\Li^\pm_{5,1} + 12\Li^\pm_{6}
    + 3\pi^2\zeta(3,1) - 3\pi^2\zeta(4) + \frac{\pi^6}{480} \nonumber\\
    &= \frac{43}{60} \log^6 2 - \frac{7}{12} \pi^2 \log^4 2 + 9 \zeta(3) \log^3 2
    + \left( 24 \LiH{4} - \frac{1}{120} \pi^4 \right) \log^2 2 \nonumber\\
    &\quad+ \left( 36 \LiH{5} - \pi^2 \zeta(3) \right) \log 2
    + 12 \LiH{5,1} + 24 \LiH{6} - \frac{247}{10080} \pi^6 \nonumber\\
    &= 2\lambda_6\left(\tfrac12\right) - 6\Li_{5,1}(-1) - 3 \zeta(3)^2 - \frac{451}{10080}\pi^6. \label{eq:ls61}
  \end{align}
  In the first equality, the term $\pi^6/480$ is the one corresponding to $n=0$
  in \eqref{eq:gfsum}.  Similarly, we find
  \begin{align}
    -\LsD{6}{2}{\pi}
    % &= \frac45 \pi \log^5 2 - \frac23 \pi^3 \log^3 2
    % + \frac{21}{2} \pi \zeta(3) \log^2 2 \nonumber\\
    % &+ 24 \LiH{4} \pi \log 2 + 24 \LiH{5} \pi - \frac{189}{16} \pi \zeta(5)
    % - \pi^3 \zeta(3) \label{eq:ls62}\\
    &= 4\pi \lambda_5\left(\tfrac12\right) - \pi^3 \zeta(3) - \frac{189}{16}\pi\zeta(5), \label{eq:ls62}\\
    -\LsD{6}{3}{\pi}
    % &= \frac{1}{10} \log^6 2 + \frac{1}{3} \pi^2 \log^4 2 + 4 \zeta(3) \log^3 2
    % - \frac{31}{60} \pi^4 \log^2 2 \nonumber\\
    % &\quad- \left( 24 \LiH{5} - \frac{3}{4} \zeta(5) - \frac{17}{2} \pi^2 \zeta(3) \right) \log 2 \nonumber\\
    % &\quad+ 24 \LiH{5,1} - 48 \LiH{6} + 12 \LiH{4} \pi^2 - \frac{17}{240} \pi^6 \\
    &= 6\pi^2\lambda_4\left(\tfrac12\right) - 12\Li_{5,1}(-1) - 6 \zeta(3)^2 - \frac{187}{1680}\pi^6, \label{eq:ls63}\\
    -\LsD{6}{4}{\pi} &= -\frac{45}{2} \pi \zeta(5) + 3 \pi^3 \zeta(3), \label{eq:ls64}
  \end{align}
  as well as
  \begin{align}
    -\LsD{7}{3}{\pi} &= \frac{9}{35} \log^7 2 + \frac{4}{5} \pi^2 \log^5 2 + 9 \zeta(3) \log^4 2
    - \frac{31}{30} \pi^4 \log^3 2 \nonumber\\
    &\quad- \left( 72 \LiH{5} - \frac{9}{8} \zeta(5) - \frac{51}{4} \pi^2 \zeta(3) \right) \log^2 2 \nonumber\\
    &\quad+ \left( 72 \LiH{5,1} - 216 \LiH{6} + 36 \pi^2 \LiH{4} \right) \log 2 + 72 \LiH{6,1} \nonumber\\
    &\quad - 216 \LiH{7} + 36 \pi^2 \LiH{5} - \frac{1161}{32} \zeta(7)
    - \frac{375}{32} \pi^2 \zeta(5) + \frac{1}{10} \pi^4 \zeta(3) \nonumber\\
    &= 6\pi^2\lambda_5\left(\tfrac12\right) + 36\Li_{5,1,1}(-1) - \pi^4\zeta(3)
    - \frac{759}{32}\pi^2\zeta(5) - \frac{45}{32}\zeta(7). \label{eq:ls73}
  \end{align}
  Note that in each case the monomials in $\LsD{n}{k}{\pi} $ are of total order
  $n$ --- where $\pi$ is order one, $\zeta(3)$ is order three and so on.
  \qede
\end{example}

\begin{remark}A purely real form of Theorem \ref{thm:gfb} is the following:
  \begin{equation}\label{eq:realgf}
    \int _{0}^{\pi }\! \left( 2\sin \frac\theta 2 \right)^{x}{\e^{\theta y}}{d\theta}
    =\sum _{n=0}^{\infty }{\frac {
 \left( -1 \right) ^{n}{x\choose n} \left( y \left(  \left( -1
 \right) ^{n}{\e^{\pi y}}-\cos  \frac{\pi x}2 \right) - \left( n-\frac x2\right) \sin  \frac{\pi x}2
 \right)}{\left( n-\frac x2 \right) ^{2}+{y}^{2}} . }
\end{equation}
  One may now also deduce one-variable generating functions from \eqref{eq:realgf}.
  For instance,
  \begin{equation}
    \sum_{n=0}^\infty \LsD{n+2}{1}{\pi} \frac{\lambda^n}{n!}
    = \sum_{n=0}^\infty \binom{\lambda}{n}
    \frac {(-1)^n \cos\frac{\pi\lambda}{2} - 1}{ \left(n -\frac\lambda2 \right)^2},
  \end{equation}
 and we may again now extract individual values. \qede
\end{remark}

\subsection[Hypergeometric evaluation of Ls(pi/3)]{Hypergeometric evaluation of $\Ls{n}{\pi/3}$}

We close this section with an alternative approach to the evaluation of
$\Ls{n}{\pi/3}$ complementing the one given in Section \ref{sec:lspi3}.

\begin{theorem}[Hypergeometric form of $\Ls{n}{\frac\pi3}$]\label{thm:l3h}
  For nonnegative integers $n$,
  \begin{align}\label{eq:lspi3series}
    \frac{(-1)^{n+1}}{n!} \Ls{n+1}{\frac\pi3} &= \pFq{n+2}{n+1}{\left\{\frac12\right\}^{n+2}}{\left\{\frac32\right\}^{n+1}}{\frac14}
    = \sum_{k=0}^\infty \frac{2^{-4k}}{(2k+1)^{n+1}} \binom{2k}{k}.
  \end{align}
  Consequently,
  \begin{align*}
    -\sum_{n=0}^\infty \Ls{n+1}{\frac\pi3} \frac{s^n}{n!} &=\frac{1}{s+1}\,\pFq21{\frac12,\frac{s}{2}+\frac12}{\frac s2+\frac32}{\frac 14}
    = \sum_{k=0}^\infty \frac{2^{-4k}}{2k+1+s} \binom{2k}{k}.
  \end{align*}
\end{theorem}

\begin{proof}
  We compute as follows:
  \begin{align*}
    -\Ls{n+1}{\frac\pi3} &= \int_0^{\pi/3} \log^n\left( 2 \sin\frac{\theta}{2} \right) \id\theta \\
    &= \int_0^1 \frac{\log^n(x)}{\sqrt{1-x^2/4}} \id x \\
    &= \sum_{k=0}^\infty 2^{-4k} \binom{2k}{k} \int_0^1 x^{2k} \log^n(x) \id x.
  \end{align*}
  The claim thus follows from
  \begin{align*}
    \int_0^1 x^{s-1} \log^n(x) \id x = \int_0^\infty (-x)^n e^{-sx} \id x
    = \frac{(-1)^n \Gamma(n+1)}{s^{n+1}}
  \end{align*}
  which is a consequence of the integral representation of the gamma function.
\end{proof}

Observe that the sum in \eqref{eq:lspi3series} converges very rapidly and so is
very suitable for computation.  Also, from Example \ref{ex:lspi3} we have
evaluations --- some known --- such as
\[\sum_{k=0}^\infty \frac{2^{-4k}}{(2k+1)} \binom{2k}{k} = \frac{\pi}3\]
and
\[\sum_{k=0}^\infty \frac{2^{-4k}}{(2k+1)^{2}} \binom{2k}{k} = \mcl{2}.\]

\begin{remark}
  As outlined in \cite{dk1-eps}, the series \eqref{eq:lspi3series}
  combined with \eqref{eq:lsbin} can also be used to produce rapidly-convergent
  series for certain multi zeta values including $\zeta(5,3)$, $\zeta(7,3)$ and
  $\zeta(3,5,3)$.
  \qede
\end{remark}

\section{Log-sine evaluations of multiple Mahler measures}

We first substantiate that we can recursively determine $\mu_k(1+x+y_*)$ from equation
\eqref{eq:mukb} as claimed.

\subsection[Evaluation of mu(1+x+y*)]{Evaluation of $\mu_k(1+x+y_*)$}
\label{sec:mu1xys}

Substituting the values given in Example \ref{ex:lspi3}  and Example \ref{ex:pi} into equation
\eqref{eq:mukb}  we obtain the following multiple Mahler evaluations:

\begin{example}[Values of $\mu_k(1+x+y_*)$]\label{ex:mu}
  We have
  \begin{align}
    \mu_1(1+x+y_*) &= \frac1\pi\mcl{2}, \label{ex:cl1}\\
    \mu_2(1+x+y_*) &= \frac{\pi^2}{54}, \label{ex:cl2}\\
    \mu_3(1+x+y_*) &= \frac{9}{2\pi}\mcl{4} - \zeta(3), \\
    \mu_4(1+x+y_*) &= \frac6\pi\mgl{4,1} - \frac{\pi^4}{4860}, \\
    \mu_5(1+x+y_*) &= \frac{135}{2\pi}\mcl{6} - 15\zeta(5) - \frac{5}{18}\pi^2\zeta(3), \\
    \mu_6(1+x+y_*) &= \frac{135}{\pi}\mgl{6,1} + 15\zeta(3)^2 - \frac{943}{40824}\pi^6,
  \end{align}
  and the like. The first is again a form of Smyth's result \eqref{eq:mu1xy}.
  \qede
\end{example}

\begin{remark}\label{rk:mu1xys}
  Note that we may rewrite the multiple Mahler measure $\mu_k(1+x+y_*)$ as
  follows:
  \begin{equation}\label{eq:mmks}
    \mu_k(1+x+y_*)
    = \mu(\underbrace{1+x,\dots,1+x}_{k-1},1+x+y).
  \end{equation}
  This is easily seen from Jensen's formula \eqref{eq:jensen}. Indeed, using
  \eqref{eq:jensen} the left-hand side of \eqref{eq:mmks} becomes
  \begin{align*}
    \mu_k(1+x+y_*)
    &= \int_0^1 \cdots \int_0^1 \prod_{j=1}^k \log \left| 1+ e^{2\pi i s} + e^{2\pi i t_j} \right|
    \md s \md t_1 \cdots \md t_k \\
    &= \int_0^1 \left[ \int_0^1 \log \left| 1+ e^{2\pi i s} + e^{2\pi i t} \right| \md t \right]^k \md s \\
    &= \int \log^k \left| 1+ e^{2\pi i s} \right| \md s
  \end{align*}
  where the last integral is over $0\le s\le1$ such that $| 1+ e^{2\pi i s}|
  \ge 1$.  The same integral is obtained when applying \eqref{eq:jensen} to the
  right-hand side of \eqref{eq:mmks}.
  \qede
\end{remark}

\subsection[Evaluation of mu(1+x+y*+z*)]{Evaluation of $\mu_k(1+x+y_*+z_*)$}
\label{sec:mu1xyszs}

We next follow a similar course for multiple Mahler measures built from
$1+x+y+z$ to that given for $\mu_k(1+x+y_*)$ in Section~\ref{sec:mu1xys}.
Analogous to \eqref{eq:defmu1xys} we define:
\begin{equation}\label{eq:defmu1xyszs}
  \mu_k(1+x+y_*+z_*) := \mu(1+x+y_1+z_1, \ldots, 1+x+y_k+z_k).
\end{equation}

Working as in \eqref{eq:mu1xyz} we may write
\begin{equation*}
  \mu_k(1+x+y_*+z_*) = \frac{1}{\pi} \int_0^{\pi}
  \left[ \frac{1}{\pi} \int_0^\pi \max\left\{\log\left( 2\sin \frac \theta
  2\right),\log\left(2 \sin \frac{\sigma}{2}\right) \right\} \id\sigma \right]^k \md\theta.
\end{equation*}
We observe that the inner integral with respect to $\sigma$ evaluates
separately, and on recalling that $\Ls{2}{\theta} = \Cl{2}{\theta}$ and
$\Cl{2}{\pi} = 0$ we arrive at:

\begin{theorem}
  For all positive integers $k$, we have
  \begin{equation}\label{eq:mu3k}
    \mu_k(1+x+y_*+z_*) = \frac{1}{\pi^{k+1}} \int_0^\pi
    \left(\theta \log\left( 2\sin \frac \theta 2\right)
    + \Cl{2}{\theta} \right)^k\id \theta.
  \end{equation}
\end{theorem}

\begin{example}[Values of $\mu_k(1+x+y_*+z_*)$]\label{ex:nu}
  Thus, for $\mu_2(1+x+y_*+z_*)$, we obtain
  \begin{equation*}
    \pi^3\, \mu_2(1+x+y_*+z_*) = -\LsD{5}{2}{\pi}
      + \int_0^\pi\! \op{Cl}_2^2(\theta) \id\theta
      + \int_0^\pi\! 2\theta\, \log\left(2\,\sin\frac\theta2\right)\,
      \Cl{2}{\theta} \id\theta.
  \end{equation*}
  Applying Parseval's equation evaluates the first integral in this equation to
  $\pi^5/180$.  Integration by parts of the second integral shows that it
  equals minus the first one.

  For $k=3$, one term is a log-sine integral and two of the terms are equal,
  but we could not completely evaluate the two remaining terms.

  Hence, from \eqref{eq:mu3k}, we have:
  \begin{align}\label{ex:klo3d}
    \mu_1(1+x+y_*+z_*) &= -\frac{2}{\pi^2}\LsD{3}{1}{\pi} = \frac{7}{2}\frac{\zeta(3)}{\pi^2}, \\
    \mu_2(1+x+y_*+z_*) &= -\frac{1}{\pi^3}\LsD{5}{2}{\pi} + \frac{\pi^2}{90}
    = \frac{4}{\pi^2}\Li_{3,1}(-1) + \frac{7}{360}\pi^2, \\
    \mu_3(1+x+y_*+z_*) &= \frac{2}{\pi^4}\int_0^\pi \op{Cl}_2^3(\theta) \id\theta
    + \frac{3}{\pi^4}\int_0^\pi \theta^2\log^2 \left(2 \sin \frac \theta 2 \right)
    \Cl{2}{\theta} \id\theta \nonumber\\
    &\quad-\frac{1}{\pi^4}\LsD{7}{3}{\pi}.\label{eq:m3i}
  \end{align}
  The first of these is a form of \eqref{eq:mu1xyz} which originates with Smyth
  and Boyd \cite{boyd}.  The relevant log-sine integrals have been discussed in
  Section \ref{sec:lsxpi}. In particular, $\LsD{5}{2}{\pi}$ and
  $\LsD{7}{3}{\pi}$ have been evaluated in \eqref{eq:ls52} and \eqref{eq:ls73}.
  
  It is possible to further reexpress the integrals in \eqref{eq:m3i} but we
  have not so far found an entirely satisfactory resolution.
  \qede
\end{example}

\subsection[Evaluation of mu(1+x,...,1+x,1+x+y+z)]{Evaluation of $\mu(1+x,\ldots,1+x,1+x+y+z)$}

Recall from Remark~\ref{rk:mu1xys} that the multiple Mahler measure
$\mu_k(1+x+y_*)$ can be rewritten as $\mu(1+x,\dots,1+x,1+x+y)$ with the term
$1+x$ repeated $k-1$ times. This is not possible for $\mu_k(1+x+y_*+z_*)$ which
is distinct from $\mu(1+x,\ldots,1+x,1+x+y+z)$ which we study next.

Applying Jensen's formula as in \eqref{eq:mu1xyz} for $k=0,1,2,\ldots$ we
obtain \eqref{eq:m1k3a}  below. Then \eqref{eq:m1k3b} follows on integrating by
parts.

\begin{theorem}\label{thm:m1k3}
  For all nonnegative integers $k$ we have:
  \begin{align}\label{eq:m1k3a}
    \mu(\smash{\underbrace{1+x,\ldots,1+x}_k}&,1+x+y+z) \nonumber\\
    &= -\frac1{\pi^2}\LsD{k+3}{1}{\pi}
    +\frac1{\pi^2}\,\int_0^\pi \Ls{2}{\theta}
    \log^{k} \left(2 \sin \frac \theta 2  \right)\id \theta \\
    &= -\frac1{\pi^2}\LsD{k+3}{1}{\pi}
    -\frac1{\pi^2}\,\int_0^\pi \Ls{k+1}{\theta}
    \log \left(2 \sin \frac \theta 2  \right)\id \theta.\label{eq:m1k3b}
  \end{align}
\end{theorem}

\begin{example}
  Equation \eqref{eq:m1k3b} recovers \eqref{eq:mu1xyz} when $k=0$ since
  $\Ls{1}{\theta} = -\theta$.  Setting $k=1$ in \eqref{eq:m1k3b}  we obtain
  \begin{align}\label{ex:1k21}
    \mu(1+x,1+x+y+z) &= -\frac1{\pi^2}\LsD{4}{1}{\pi}\nonumber
    +\frac1{\pi^2}\,\int_0^\pi \Cl{2}{\theta}\,
    \log \left(2 \sin \frac \theta 2  \right) \id\theta \nonumber\\
    &= -\frac1{\pi^2}\LsD{4}{1}{\pi}
    -\frac{1}{2\pi^2} \op{Cl}_2^2(\pi) \nonumber\\
    &= -\frac1{\pi^2}\LsD{4}{1}{\pi}= \frac{2}{\pi^2}\lambda_4\left(\tfrac12\right) - \frac{19}{720}\pi^2
  \end{align}
  on again using $\Ls{2}{\theta} = \Cl{2}{\theta}$ and  $\Cl{2}{\pi} = 0$.
  The final evaluation was given in \eqref{eq:ls41} of Example \ref{eg:lslow}.
  For $k=2$ we have
  \begin{align*}
    \mu(1+x,1+x,1+x+y+z) &= -\frac1{\pi^2}\LsD{5}{1}{\pi}\nonumber
    +\frac1{\pi^2}\,\int_0^\pi \Ls{2}{\theta}
    \log^{2} \left(2 \sin \frac \theta 2  \right)\id \theta \nonumber\\
    &= -\frac1{\pi^2}\LsD{5}{1}{\pi}
    -\frac{2}{3\pi^2}\lambda_5\left(\tfrac12\right)+  \frac{155}{32\pi^2}\,\zeta(5),
  \end{align*}
  where the last integral was found via PSLQ. This agrees with the more
  complicated version conjectured in \cite{kalmykov-eps}.  We may use
  \eqref{eq:ls51} of Example \ref{eg:lslow} to arrive at
  \begin{align}\label{ex:2k21}
    \mu(1+x,1+x,1+x+y+z)
    &=
    \frac{4}{3\pi^2}\lambda_5\left(\tfrac12\right) - \frac{3}{4}\zeta(3) + \frac{31}{16\pi^2}\zeta(5).
  \end{align}
  For $k=3$, things are more complicated as is suggested by  \eqref{eq:ls61}.
  \qede
\end{example}

\section{Moments of random walks}\label{sec:walks}

The $s$-th moments of an $n$-step uniform random walk are given by
\begin{equation}
  W_n (s) = \int_0^1 \ldots \int_0^1 \left| \sum_{k = 1}^n e^{2 \pi
  i t_k} \right|^s \md t_1 \cdots \md t_n
\end{equation}
and their relation with Mahler measure is observed in \cite{bswz-densities}.
In particular,
\[ W_n' (0)  = \mu (1 + x_1 + \ldots + x_{n-1}) \]
with the cases $n=3$ and $n=4$ given in \eqref{eq:mu1xy} and
\eqref{eq:mu1xyz} respectively. The cases $n=5$ and $n=6$ are
discussed in  \eqref{eq:vil1} and \eqref{eq:vil2} respectively.
Higher derivatives of $W_n$ correspond to higher Mahler measures:
\begin{equation}\label{eq:diffW}
  W_n^{(k)}(0) = \mu_k (1 + x_1 + \ldots + x_{n-1}).
\end{equation}
More general moments corresponding to other Mahler measures were introduced in
\cite{aka} and studied in \cite{klo} as \emph{zeta Mahler measures}.

\subsection[Evaluation of mu(1+x)]{Evaluation of $\mu_k(1+x)$}
\label{sec:mu1x}

Equipped with the results of the first section, we may now fruitfully revisit
another recent result which is concerned with the evaluation of $W_2^{(k)}(0) =
\mu_k(1+x)$.

A central evaluation in \cite[Thm. 3]{klo} is:
\begin{equation}\label{eq:klo1}
  \mu_k(1+x) = (-1)^k k! \sum_{n=1}^\infty \frac 1{4^n}
    \sum_{b_j \ge 2, \sum b_j=k}\zeta(b_1,b_2,\ldots, b_n).
\end{equation}
We note that directly from the definition and an easy change of variables
\begin{equation}\label{eq:muk1x}
  \mu_k(1+x) = -\frac1\pi \Ls{k+1}{\pi}.
\end{equation}
Hence, we have closed forms such as provided by Example \ref{ex:pi}.

\begin{example}
  For instance,
  \begin{align}\label{ex:klo3b}
    -\mu_5(1+x) &= \frac{45}{2}\zeta(5) + \frac{5}{4} \pi^2\zeta(3), \\
     \mu_6(1+x) &= \frac{45}{2} \zeta^2(3) + \frac{275}{1344} \pi^6.
  \end{align}
  These are derived more elaborately in \cite[Ex. 5]{klo} from the right of
  equation \eqref{eq:klo1}.
  \qede
\end{example}

We have, inter alia, evaluated the multi zeta value sum on the right of
equation \eqref{eq:klo1} as a simple log-sine integral.

Also, note that the evaluation $W_2(s) = \binom{s}{s/2}$,
\cite{bswz-densities}, in combination with \eqref{eq:muk1x} thus explains and
proves the generating function \eqref{eq:lsatpiogf}.

\subsection[A generating function for mu(1+x+y)]{A generating function for $\mu_k(1+x+y)$}

The evaluation of the Mahler measures $W_3'(0) = \mu(1+x+y)$ and $W_4'(0) =
\mu(1+x+y+z)$ is classical and was discussed in Example~\ref{eg:classical}.

The derivatives $W_3''(0) = \mu_2(1+x+y)$ and $W_4''(0) = \mu_2(1+x+y+z)$ were
evaluated using explicit forms for  $W_3 (s)$ and $W_4(s)$ in
\cite[\S6]{bswz-densities}. For example,
\begin{align}\label{w3d2}
  W_3''(0) & =\frac{\pi^2}{12} - \frac{4 \log 2}{\pi}\mcl{2}
  - \frac{4}{\pi} \sum_{n=0}^\infty \frac{\binom{2n}{n}}{4^{2n}}
    \frac{\sum_{k=0}^n \frac 1{2k+1}}{(2n+1)^2}.
\end{align}
We shall revisit these two Mahler measures in \eqref{eq:mu21xyLs} and
\eqref{eq:mu21xyz} of Sections \ref{sec:mu21xy} and \ref{sec:mu21xyz}.

As a consequence of the study of random walks in \cite{bswz-densities} we
record the following generating function for $\mu_k(1+x+y)$ which follows from
\eqref{eq:diffW} and the hypergeometric expression for $W_3$ in \cite[Thm.
10]{bswz-densities}. There is a corresponding expression, using a single
Meijer-$G$ function, for $W_4$ (i.e., $\mu_m(1+x+y+z)$) given in \cite[Thm.
11]{bswz-densities}.

\begin{theorem}\label{thm:hyper}
  For complex $|s|<2$, we may write
  \begin{equation}\label{singleh}
    \sum_{m=0}^\infty \mu_m(1+x+y)\frac{s^m}{m!} = W_3(s)
    =\frac {\sqrt{3}}{2\pi}\, 3^{s+1}
    \,\frac{\Gamma(1+s/2)^2}{\Gamma(s+2)}\,
      \pFq32{\frac{s+2}2, \frac{s+2}2,\frac{s+2}2}{1,\frac{s+3}2}{\frac14}.
  \end{equation}
\end{theorem}

The particular measure $\mu_2(1+x+y)$ will be investigated in Section
\ref{sec:mu21xy}. The general case $\mu_m(1+x+y)$ is studied in \cite{logsin2}.

\subsection[Evaluation of mu2(1+x+y)]{Evaluation of $\mu_2(1+x+y)$}
\label{sec:mu21xy}

\begin{example}\label{ex:klo1b}
  A purported evaluation given in \cite{klo} is:
  \begin{equation}\label{eq:klo2}
    \mu_2(1+x+y)  = \mu_2(1+x+y) \stackrel{?}{=} \frac{5}{54}\, \pi^2
    = 5\mu_2(1+x+y_*)
  \end{equation}
  where the last equality follows from \eqref{ex:cl2}. However, we are able to
  numerically disprove \eqref{eq:klo2}.\footnote{There are two errors in the
  proof given in \cite[Theorem 11]{klo}. A term is dropped between lines 8 and 9 of the
  proof and the limits of integration are wrong after changing $s(1-s)$ to
  $t$.} Indeed, we find $\mu_2(1+x+y)\approx0.419299$ while $\frac{5}{54}\,
  \pi^2\approx0.913852$.
  \qede
\end{example}

We note that for integer $k \ge 1$ we do have
\begin{equation}\label{eq:mu2K}
  \mu_k(1+x+y) =\frac{1}{4\pi^2}\,\int_0^{2\pi} \md\theta
    \int_0^{2\pi} \left(\Re\log\left(1-2\sin(\theta)\e^{i\,\omega}\right)\right)^k\id \omega,
\end{equation}
directly from the definition and some simple trigonometry, since $\Re\log
=\log | \cdot|$. We revisit Example \ref{ex:klo1b} in the next result, in which
we evaluate $\mu_2(1+x+y)$ as a log-sine integrals as well as in terms of
polylogarithmic constants.

\begin{theorem}\label{thm:mu21xy}
  We have
  \begin{align}\label{eq:mu2A}
    \mu_2(1+x+y) &= \frac{24}{5\pi}\Ti_3 \left(\frac 1{\sqrt{3}} \right)
    + \frac{2\log3}\pi\mcl{2} - \frac{\log^2 3}{10} - \frac{19\pi^2}{180} \\
    &= \frac{\pi^2}4 + \frac3\pi\Ls3{\frac{2\pi}{3}}. \label{eq:mu21xyLs}
  \end{align}
\end{theorem}

\begin{remark}
  We note that
  \begin{equation*}
    \Ls{3}{\frac{2\pi}{3}}
    = -\int_0^{\pi/3} \log^2 \left(2 \cos\frac\theta 2\right) \id\theta
  \end{equation*}
  and that these log-cosine integrals have fewer explicit closed forms. Using
  the results of \cite{logsin-evaluations} to evaluate log-sine integrals in
  polylogarithmic terms we find that
  \begin{equation}\label{eq:ls3at2pi3}
    \Ls{3}{\frac{2\pi}{3}}
    = -\frac{13}{162}\pi^3 - 2\Gl{2,1}{\frac{2\pi}{3}}.
  \end{equation}
  In fact, this is automatic if we employ the provided implementation.  Theorem
  \ref{thm:mu21xy} thus also gives a reduction of $\Gl{2,1}{\frac{2\pi}{3}}$ to
  one-dimensional polylogarithmic constants.
  \qede
\end{remark}

A preparatory result is helpful before proceeding to the proof of Theorem
\ref{thm:mu21xy}.

\begin{proposition}[A dilogarithmic representation]\label{prop:mu2}
  We have:
  \begin{enumerate}[(a)]
    \item\label{propa}
      \begin{equation}\label{eq:z2}
        \frac 2\pi\,\int_{0}^{\pi} \Re\Li_2\left(4\sin^2\theta\right) \md\theta = 2\zeta(2).
      \end{equation}
     \item\label{propb}
       \begin{equation}\label{eq:mu2}
        \mu_2(1+x+y) =
        \frac1{36}{\pi }^{2}+\frac 2\pi\,\int_{0}^{\pi/6} \Li_2\left(4\sin^2\theta\right) \md\theta.
      \end{equation}
  \end{enumerate}
\end{proposition}

\begin{proof}
  For \eqref{propa} we define $\tau(z):= \frac 2\pi\,\int_{0}^{\pi}
  \Li_2\left(4z\sin^2\theta\right) \md\theta.$ This is an analytic function of
  $z$. For $|z|<1/4$ we may use the original series for $\Li_2$ and  expand
  term by term using Wallis' formula to derive
  \begin{align*}
    \tau(z) &= \frac 2\pi\,\sum_{n\ge 1}\frac{(4z)^n}{n^2}\int_{0}^{\pi}\,\sin^{2n}\theta\id\theta
    = 4z\,\pFq43{1, 1, 1, \frac 32}{2, 2, 2}{4z} \\
    &= 4\Li_2\left( \frac12-\frac12\sqrt{1-4z} \right) - 2 \log\left( \frac12+\frac12\sqrt{1-4z} \right)^2.
  \end{align*}
  The final equality can be obtained in \emph{Mathematica} and then verified by
  differentiation.  In particular, the final function provides an analytic
  continuation and so  we obtain $\tau(1)=2\zeta(2) + 4i\mcl{2}$ which yields
  the assertion.

  For \eqref{propb}, commencing  much as in \cite[Thm. 11]{klo}, we write
  \begin{equation*}
    \mu_2(1+x+y) = \frac{1}{4\pi^2}\,\int_{-\pi}^{\pi} \int_{-\pi}^{\pi}
    \Re\log\left(1-2\sin(\theta)\e^{i\,\omega}\right)^2 \id\omega\id\theta.
  \end{equation*}
  We consider the inner integral
  $\rho(\alpha):=\int_{-\pi}^{\pi}\left(\Re\log\left(1-\alpha\,\e^{i\,\omega}\right)\right)^2\id \omega$ with $\alpha := 2\sin\theta$.  For
  $|\theta| \le \pi/6$ we directly apply Parseval's identity to obtain
  \begin{equation}\label{eq:mrhoA}
    \rho(2\sin \theta) = \pi \Li_2\left(4\sin^2\theta\right).
  \end{equation}
   In the remaining case we write
  \begin{align}\label{eq:mrhoB}
    \rho(2\sin\theta) &=
    \int_{-\pi}^{\pi}\left\{\log|\alpha| + \Re\log \left(1-\alpha^{-1}\,\e^{i\,\omega}\right)\right\}^2 \id\omega \nonumber\\
    &= 2\pi\,\log^2 |\alpha|- 2\log|\alpha|\int_{-\pi}^{\pi}\log \left|1-\alpha^{-1}\,\e^{i\,\omega}\right|\id \omega
    + \pi \Li_2\left(\frac{1}{4\sin^2\theta}\right) \nonumber\\
    &= 2\pi\,\log^2 |2\sin\theta| + \pi \Li_2\left(\frac{1}{4\sin^2\theta}\right),
  \end{align}
  where we have appealed to Parseval's and Jensen's formulas.  Thus,
  \begin{equation}\label{eq:mrhoC}
    \mu_2(1+x+y) = \frac 1\pi\,\int_{0}^{\pi/6} \Li_2\left(4\sin^2\theta\right) \md\theta
    + \frac 1\pi\,\int_{\pi/6}^{\pi/2} \Li_2\left(\frac{1}{4\sin^2\theta}\right) \md\theta
    + \frac {\pi^2}{54},
  \end{equation}
  since $\frac 2\pi\int_{\pi/6}^{\pi/2}\log^2 \alpha\id \theta = \mu_2(1+x+y_*) =
  \frac{\pi^2}{54}$.  Now, for $\alpha >1$, the functional equation in
  \cite[A2.1 (6)]{lewin} --- $\Li_2 (\alpha) +\Li_2 (1/\alpha) +\frac
  12\,\log^2 \alpha =2\zeta(2)+ i\pi \log \alpha $ --- gives:
  \begin{equation}\label{eq:fed1}
    \int_{\pi/6}^{\pi/2}\left\{\Re\Li_2 \left(4 \sin^2 \theta \right)
    + \Li_2 \left(\frac1{4 \sin^2 \theta} \right)\right\}\id \theta
    = \frac{5}{54} \pi^3.
  \end{equation}
  We now combine \eqref{eq:z2}, \eqref{eq:fed1} and \eqref{eq:mrhoC} to deduce
  the desired result in \eqref{eq:mu2}.
\end{proof}

We are now in a position to prove the desired evaluation of $\mu_2(1+x+y)$ as a
log-sine integral.

\begin{proof}[Proof of Theorem \ref{thm:mu21xy}]
Using Proposition \ref{prop:mu2} we have:
\begin{align}
  \mu_2(1+x+y)
  &= \frac{\pi^2}{36} + \frac2\pi \int_0^{\pi/6}\!
    \Li_2 \left(4\, \sin^2 w  \right) \id w \nonumber\\
  &= \frac{\pi ^2}{36}+ \frac2\pi\, \sum_{n \ge 1} \frac{4^n}{n^2}
    \int_0^{\pi/6} \sin^{2n} w \id w \nonumber\\
  % &= \frac{\pi^2}{36}+\frac1\pi\, \sum_{n \ge 1} \frac{1}{(2n+1)n^2}\,
    % \pFq21{\frac12,n+\frac12}{n+\frac32}{\frac14} \nonumber\\
  &= \frac{\pi^2}{36}+\frac{\sqrt {3}}\pi\,\sum_{n \ge 1}
    \frac{{2\,n-1\choose n-1}}{4^n}\,\sum _{k=n}^{\infty }
    \frac {1}{ \left( 2\,k+1 \right) {2\,k\choose k}},
\end{align}
where the last line is a consequence of the formula
\begin{equation*}
  \int _{0}^{\pi/6} \sin^{2n} w \,{\md w}
  = \frac{\sqrt {3}}2\frac{{2\,n-1\choose n-1}}{4^n}\,
  \sum _{k=n}^{\infty }{\frac {1}{ \left( 2\,k+1 \right) {2\,k\choose k}}}
\end{equation*}
given in \cite{klo}.  Hence, on using a beta-integral and then exchanging sum
and integral we obtain:
\begin{align}
  \mu_2(1+x+y)
  &= \frac{\pi^2}{36}+\frac{2\sqrt {3}}\pi\,\sum_{n \ge 1}{2\,n-1\choose n-1}
    \int_0^{1/2} \frac{t^n(1-t)^n}{1-t+t^2}\id t \nonumber \\
  &= \frac{\pi^2}{36}+\frac{2\sqrt {3}}\pi\,\int_0^{1/2}
    \sum_{n \ge 1}{{2\,n-1}\choose n-1} \frac{\left(t(1-t)\right)^n}{1-t+t^2}\id t\nonumber \\
  &= \frac{\pi^2}{36}+\frac{\sqrt {3}}\pi\,\int _{0}^{1/2}\!
    {\frac {\,2\Li_2\left(t \right) - \log^2  \left( 1-t \right) }{1-t+t^2}}{\id t}
\end{align}
where the last equality comes from evaluating the power series above.

% Now integration by parts allows evaluation to
% \begin{eqnarray*}\label{eq:mu2d}\mu_2(1+x+y) &=& \frac 13\,
% \ln^2 2-\frac {31}{324}\,{\pi }^{2}+\frac2\pi\,\sqrt {3}\int _{0}^{1/2}\,
% \frac {\Li \left(t \right)}{1-t+{t}^{2}}{\id t}\\
% &+& \frac{4}{\pi}\,{\rm Cl}_2 \left( \frac 12, \frac\pi3 \right) \log  2
% +\Im\Li_3 \left(\frac{i\sqrt {3}+1}4 \right)
% \end{eqnarray*}
% \begin{eqnarray*}\label{eq:mu2d2}
% &=& \frac{10}{3\pi}\,{\rm Cl}_2 \left( \frac\pi3 \right) \log  2
% -\frac 13\, \ln^2 2-\frac {31}{324}\,{\pi }^{2} +\Im\,\Li_3 \left(\frac{i\sqrt {3}+1}4 \right)
% \\&+&\frac2\pi\,\sqrt {3}\int _{0}^{1/2}\,
% \frac {\Li \left(t \right)}{1-t+{t}^{2}}{\id t}.
% \end{eqnarray*}

Further careful integrations by parts let us use \cite[Appendix A5.3, (1)]{lewin2} to derive
\begin{align}\label{eq:mu2eb}
  \pi \mu_2(1+x+y) &= \frac{67}{324}\,{\pi }^{3}+2\mcl{2}\log 3-8\,\Im \, \Li_3 \left(i\sqrt {3}
 \right) \nonumber\\
 &\quad+ 4\Im\Li_3 \left(\frac{3+i\sqrt {3}}2 \right).
\end{align}
Next, we note that
\begin{align}\label{eq:mu2y}
  \Im\Li_3 \left(\frac{3+i\sqrt {3}}2 \right)
  &= \frac{55}{1296}\,{\pi }^{3}+\frac{5}{48}\,\pi \,\log^2 3 +\Im\Li_3 \left(\frac{3-i\sqrt {3}}6 \right),
\end{align}
while
\begin{align}\label{eq:mu2z}
  \Im\Li_3 \left(i\sqrt {3} \right)
  &= \frac{1}{16}\,{\pi }^{3}+\frac{1}{16}\,\pi \,\log^2 3 -\frac16\,\Ti_3\left(\frac1{\sqrt{3}}\right).
\end{align}
Above, we have had recourse to various reduction formulae \cite{lewin2,box3}
for higher Clausen functions to arrive at the final form.  Substituting
\eqref{eq:mu2y}, \eqref{eq:mu2z} in \eqref{eq:mu2eb}, we arrive at the asserted
result \eqref{eq:mu2A}.

A connection with the log-sine integrals is made by noting that
\begin{align}
  \Ti_3\left(\frac1{\sqrt{3}}\right)
  &= \frac58\Ls{3}{\frac{2\pi}{3}}
  - \frac12\,\Ti_2\left(\frac1{\sqrt{3}}\right)\,\log 3
  - \frac1{48}\,\pi \,\log^2 3+\frac {2}{27}\,{\pi }^{3}, \label{eq:tan-ls1}\\
  \Ti_2\left(\frac1{\sqrt{3}}\right)
  &= \frac56 \mcl{2}-\frac{\pi}{12}\,\log 3. \label{eq:tan-ls2}
\end{align}
These follow from \cite[Eqn. (44), p. 298]{lewin2} and \cite[Eqn. (18), p.
292]{lewin2} respectively. Applying  \eqref{eq:tan-ls1} and \eqref{eq:tan-ls2}
to \eqref{eq:mu2A} now yields \eqref{eq:mu21xyLs}.
\end{proof}

Finally, we observe that it is possible to take the analysis of $\mu_n(1+x+y)$
for $n \ge 3$ a fair distance. This will be detailed in the forthcoming paper
\cite{logsin2}.

\subsection[Evaluation of mu2(1+x+y+z)]{Evaluation of $\mu_2(1+x+y+z)$}
\label{sec:mu21xyz}

Paralleling the evaluation of $\mu_2(1+x+y)$ in Theorem \ref{thm:mu21xy} we now
give a closed form for $\mu_2(1+x+y+z)$ which was obtained in
\cite{bswz-densities} by quite different methods to those of
Theorem~\ref{thm:mu21xy}.

\begin{theorem}\label{thm:mu21xyz}
  We have
  \begin{align}\label{eq:mu21xyz}
    \mu_2(1+x+y+z) = \frac{12}{\pi^2}\lambda_4\left(\tfrac12\right)
    - \frac{\pi^2}{5}
  \end{align}
  where $\lambda_4$ is as defined in \eqref{eq:bw}.
\end{theorem}

\begin{proof}
  The formula
  \begin{equation*}\label{eq:2d}
    \pi^2 W_4''(0) = 24\Li_4(\tfrac12)-18\zeta(4)+21\zeta(3) \log 2
    - 6\zeta(2)\log^2 2 +\log^4 2
  \end{equation*}
  was deduced in \cite{bswz-densities}. We now observe that
  \begin{align*}\label{eq:2db}
    24\Li_4(\tfrac12)-18\zeta(4)+21\zeta(3) \log 2
     & - 6\zeta(2)\log^2 2 +\log^4 2 \nonumber\\
      &= 12\lambda_4\left(\tfrac12\right)-\frac{\pi^4}5,
  \end{align*}
  and appeal to equation \eqref{eq:diffW} for $\mu_2(1+x+y+z) = W_4''(0)$.
\end{proof}

\subsection{A conjecture of Rodriguez-Villegas}
\label{sec:vill}

Finally, we mention two conjectures concerning the Mahler measures
$\mu(1+x+y+z+w)$ and $\mu(1+x+y+z+w+v)$, contained in slightly different form
in \cite{finch}. These correspond to the moment values $W_5'(0)$ and $W_6'(0)$.

Recall that $\eta$ is the \emph{Dirichlet eta-function} given by
\begin{equation}\label{eq:eta}
  \eta(\tau) =\eta(q):=q^{1/24}\, \prod_{n=1}^\infty(1-q^n)
  = q^{1/24}\,\sum_{n=-\infty}^\infty (-1)^n q^{n(3n+1)/2}
\end{equation}
where $q=e^{2\pi i\tau}$.

The following two conjectural expressions have been put forth by
Rodriguez-Villegas:
\begin{equation}\label{eq:vil1}
  \mu(1+x+y+z+w) \stackrel{?}{=} \left(\frac{15}{4\pi^2}\right)^{5/2}\,
  \int_0^\infty \left\{\eta^3(e^{-3t})\eta^3(e^{-5t})
  +\eta^3(e^{-t})\eta^3(e^{-15t})\right\} t^3\id t
\end{equation}
and
\begin{equation}\label{eq:vil2}
  \mu(1+x+y+z+w+v) \stackrel{?}{=} \left(\frac{3}{\pi^2}\right)^{3}\,
  \int_0^\infty \,\eta^2(e^{-t})\eta^2(e^{-2t}) \eta^2(e^{-3t})\eta^2(e^{-6t})\,t^4\id t.
\end{equation}
As discussed in \cite{bswz-densities}, we have confirmed numerically that the
evaluation of $\mu(1+x+y+z+w+v)$ in \eqref{eq:vil1} holds to $600$ places.
Likewise, we have confirmed that \eqref{eq:vil2} holds to $80$ places.

\section{Conclusion}

It is reasonable to ask what other Mahler measures can be placed in log-sine
form, and to speculate as to whether the $\eta$ integrals \eqref{eq:vil1} and
\eqref{eq:vil2} can be.

As described in \cite{ls10}, it is a long standing question due to Lehmer as to
whether, for single-variable integer polynomials $P$, $\mu(P)$ can be
arbitrarily close to zero. For higher Mahler measures \cite[Thm. 7]{ls10} shows
that for $k=1,2,\ldots$ the measure $\mu_{2k+1}\left((x^n-1)/(x-1)\right)$ does
tend to zero as $n$ goes to infinity.

It was shown in \eqref{eq:mukb} that for positive integers $k$,
\begin{equation}
  \pi\,\mu(1+x+y_1,1+x+y_2, \ldots, 1+x+y_k) = \Ls{k+1}{\frac\pi3}-\Ls{k+1}{\pi}.
\end{equation}
This rapidly tends to zero with $k$ since $|\Ls{k+1}{\frac\pi3}-\Ls{k+1}{\pi}|
\le \frac{2\pi}3 \, \log^k 2$.  Can one find any natural polynomial sequences
so that $\mu(P_n,Q_n)$ tends to zero with $n$ and so generalize \cite[Thm.
7]{ls10}\,?

\vfill

\paragraph{Acknowledgements}
Thanks are due to Yasuo Ohno and Yoshitaka Sasaki for introducing us to the
relevant papers during their recent visit to CARMA. Especial thanks are due to
James Wan and David Borwein for many useful discussions.  We are also grateful
to the reviewer for several helpful suggestions.

\bibliography{logsin-refs}
\bibliographystyle{abbrveprint}

\end{document}